\documentclass[a4paper,11pt]{article}
\usepackage{amsmath}
\usepackage{amsfonts}
\usepackage{amssymb}
\usepackage{amsthm}
\usepackage{amsbsy}
\usepackage{mathrsfs}
\usepackage{latexsym}
\usepackage{graphicx} 
\usepackage[latin1]{inputenc}			% Caracteres con acentos y smbolos latinos
\usepackage[english]{babel} 			% \usepackage[spanish]{babel}
\usepackage[T1]{fontenc}				% Para incorporar simbolos acentuados y otros smbolos
\usepackage{verbatim}					% Incluye texto en fuente mono-espaciada y comentarios
\usepackage{enumerate}					% Para el manejo de listas enumeradas y encajadas
\usepackage{array}						% Para darle formato al encabezado de las tablas
\usepackage{multicol}					% Allows to shift between two and one columns anywhere
\usepackage{paralist}					% Extended list environments 
\usepackage{multirow}
\usepackage{dsfont}
\usepackage{ctable}
\usepackage[all]{xy}					% Typesetting graphs and diagrams
\usepackage[dvips,colorlinks,citecolor = blue]{hyperref}
\usepackage{url}
\usepackage[numbers]{natbib}			% Citations and references \usepackage[round]{natbib}
\usepackage[hang,splitrule]{footmisc}

\DeclareMathOperator*{\argmin}{arg\,min}

\usepackage{algorithm}
\usepackage{algorithmic}

\theoremstyle{plain}

\newtheorem{theorem}{Theorem}

\theoremstyle{definition}
\newtheorem{definition}{Definition}

\hyphenrules{nohyphenation}
\usepackage{titling} % to reduce the white space that appears before the title
\usepackage[affil-it]{authblk}
\title{\textbf{A robust algorithm for template curve estimation based on manifold embedding}}

\author[a]{Chlo\'e Dimeglio\thanks{\texttt{cd@geosys.com}}}
\author[a]{Santiago Gallón\thanks{\texttt{santiagog@udea.edu.co}}}
\author[a]{Jean-Michel Loubes\thanks{\texttt{jean-michel.loubes@math.univ-toulouse.fr}}}
\author[c]{Elie Maza\thanks{\texttt{Elie.Maza@ensat.fr}}}
\affil[a]{\small{Institut de Mathématiques de Toulouse, Universit\'e Toulouse III Paul Sabatier, Toulouse, France.}}
\affil[c]{Laboratoire G\'enomique et Biotechnologie des Fruits, UMR 990 INRA/INP-ENSAT, Universit\'e Toulouse III Paul Sabatier, Toulouse, France.}

\begin{document}
\maketitle

\begin{abstract}
\noindent This paper considers the problem of finding a meaningful template function that represents the common  pattern of a sample of curves. To address this issue, a novel algorithm based on a robust version of the isometric featuring mapping (Isomap) algorithm is developed. Assuming that the functional data lie on an intrinsically low-dimensional smooth manifold with unknown underlying  structure, we propose an approximation of the geodesic distance. This approximation is used to compute the corresponding empirical Fr\'{e}chet median function, which provides an intrinsic estimator of the template function.  Unlike the Isomap method, the algorithm has the advantage of being parameter free and easier to use. Comparisons with other methods, with both simulated and real datasets, show that the algorithm works well and outperforms these methods.\vskip .1in
\noindent \textbf{Key words: } Fr\'{e}chet median; functional data analysis;  Isomap.
%\noindent \textbf{Subject Class. MSC-2010:}\mysubjclass
\end{abstract}

\section{Introduction}
\label{section1}

Nowadays, experiments where the outcome constitutes a sample of functions $\{f_i(t)\colon t\in\mathcal{T}\subset\mathbb{R},\,i = 1,\ldots,n\}$ are more and more frequent. Such kind of functional data are now commonly encountered in speech signal recognition in engineering, growth curves analysis in biology and medicine, microarray experiments in molecular biology and genetics, expenditure and income studies in economics, just to name a few.

However, extracting the information conveyed by all the curves is a difficult task. Indeed  when finding a meaningful representative function that characterizes the common behavior of the sample, capturing its inner characteristics (as trends, local extrema and inflection points), a major difficulty comes from the fact that usually there are both amplitude (variation on the $y$-axis) and phase (variation on the $x$-axis) variations with respect to the common pattern, as pointed out in  \citet{Ramsay-Li-98}, \citet{Ramsay-Silverman-05}, or \citet{Vantini-12} for instance. Hence, in the two last decades, there has been a growing interest for statistical methodologies and algorithms to remove the phase variability and recover a single template conveying all the information in the data since the classical cross-sectional mean is not a good representative of the data  (see for instance \citet{Kneip-Gasser-92}). 

Two different kinds of methods have been developed for template function estimation. The first group  relies on the assumption that there exists a mean pattern from which all the observations differ, i.e an unknown function $f$ such that each observed curve is given by $f_{i}(t) = f\circ h_{i}(t)$, where $h_i$ are deformation functions. Hence finding this patten is achieved by aligning all curves $f_i$. This  method is known as \emph{curve registration}. In this direction, various curve registration methods have been proposed using different  strategies. When the warping operator is not specified, we refer  for instance to \citet{Kneip-Gasser-92}, \citet{Wang-Gasser-97}  \citet{Kneip-00},  \citet{James-07}, \citet{Tang-Muller-08}, and \citet{Kneip-Ramsay-08} or \citet{Dupuy-Loubes-Maza-11}. When a parametric model for the deformation is chosen, the statistical problem requires a semi-parametric approach through a self-modeling regression framework $f_{i}(t) = f(t,\theta_i)$ (see \citet{Kneip-Gasser-88}), where all functions are deduced with respect to the template $f$ by mean a finite-dimensional individual parameter vector $\theta_i$. This point of view is also followed in  \citet{Silverman-95}, \citet{Ronn-01}, \citet{Gamboa-Loubes-Maza-07}, \citet{Castillo2009},  \citet{Bigot2010b} and \citet{Trigano-Isserles-Ritov-11}.

The second category of methods do not assume any deformation model for the individual functions. The purpose is to select a curve which is assumed to be located at the {\it center} of the functions and estimate it directly from the data without stressing any particular curve. This is achieved for instance by \citet{Lopez-Romo-09} and \citet{ArribasGil-Romo-12} estimating the template based on the concept of depth for functional data as measure of centrality of the sample.

In this paper, we propose an alternative way based on the ideas of manifold learning theory. We assume that the observed functions can be modeled as variables with values in a manifold $\mathcal{M}$ with an unknown geometry. Although the manifold is unknown, the key property is that its underlying geometric structure is contained in the sample of observed curves so that the geodesic distance can be reconstructed directly from the data. The template curve estimation is then equivalent to consider a location measure of the data with respect to this geodesic distance, hence approximating the Fr\'echet mean or median of the data.
Recently, \citet{Chen-Muller-12} have also adopted a similar methodology appealing to the nonlinear manifold representation for functional data.
Several algorithms  have been developed over the last decade in order to reconstruct the natural embedding of data onto a manifold. Some of these are, for instance, the Isometric featuring mapping $-$Isomap$-$ (\citet{Tenenbaum-2000}), Local Linear Embedding $-$LLE$-$ (\citet{Roweis-Sau-00}), Laplacian Eigenmap (\citet{Belkin-Niyogi-03}), Hessian Eigenmap (\citet{Donoho-Grimes-03}), Diffusion maps
(\citet{Coifman-Lafon-06}), Local Tangent Space Alignment $-$LTSA$-$ (\citet{Zhang-Zha-04}), among others. In the following, we propose a robust version of the Isomap algorithm devoted to functional data, less sensitive to outliers and easier to handle. The performance of the algorithm is evaluated both on simulations and real data sets.

This article is organized as follows. The frame of our study is described in Section~\ref{s:section1}. Section~\ref{section2} is devoted to  the robust modification of the Isomap algorithm proposed to the metric construction of the approximated geodesic distance based on the observed curves. In Section~\ref{section3} we analyze the template estimation problem in a shape invariant model, showing that this issue can be solved using the manifold geodesic approximation procedure. In Section~\ref{section4}, the performance of our algorithm is studied using simulated data. In Section~\ref{section5}, several applications on real functional data sets are performed. Some concluding remarks are given in Section~\ref{section6}.

\section{Template estimation with a manifold embedding framework} \label{s:section1}
Consider discrete realizations of  functions $f_{i}$ observed at time $t_{j}\in \mathcal{T}$, with $\mathcal{T}$ a bounded interval of $\mathbb{R}$. % We assume that all curves are observed at the same time with the same occurrence, i.e. $t_{ij}=t_{j}$ and $j=1,\dots,m$. 
Set $X_{i} = \{ f_{i}(t_{j}),\:j = 1,\ldots,m \} \in \mathbb{R}^m$ for $i = 1,\ldots,n$. We assume that the data have a common structure which can be modeled as a manifold embedding. Hence the sample $\mathcal{E}=\left\{X_1,\dots,X_n\right\}$   consists of  i.i.d random variables sampled from a law $Q\in\mathcal{M}$ ,where  $\mathcal{M}$ is an unknown connected smooth submanifold of $\mathbb{R}^m$, endowed with the geodesic distance $\delta$ induced by the Riemannian metric $g$ on $\mathcal{M}\subset \mathbb{R}^m$ (see for instance \citet{Carmo1992}).

Under this geometrical framework, the statistical analysis of the curves   should be carried out carefully, using the intrinsic geodesic distance and not the Euclidean distance, see for instance \citet{Pennec2006} . In particular, an extension of the usual notion of central value from Euclidean spaces to arbitrary manifolds is based on the Fr\'{e}chet function, defined by
\begin{definition}[Fr\'{e}chet function]
Let $(\mathcal{M}, \delta)$ be a metric space and let $\alpha>0$ be a given real number. For a given probability measure $Q$ defined on the Borel $\sigma$-field of $\mathcal{M}$, the Fr\'{e}chet function of $Q$ is given by
\[
F_\alpha(\mu)=\int_{\mathcal{M}}\delta^{\alpha}(X,\mu)Q(\text{d}x),\qquad\mu\in\mathcal{M}.
\]
\end{definition}

For $\alpha=1$ and $\alpha=2$, the minimizers of $F_\alpha(\mu)$, if there exist, are called the Fr\'echet (or intrinsic) median and mean respectively.  Following \citet{Koenker2006}, in this paper we will particularly deal with the intrinsic median, denoted by $\mu_{I}^{1}(Q)$ to obtain a robust estimate for the template function $f\in\mathcal{M}$. Hence define  the corresponding empirical intrinsic median as
\begin{equation}\label{intrinsicStatistic0}
\widehat{\mu}_I^1=\argmin_{\mu\in\mathcal{M}}\sum_{i=1}^n\delta\left(X_i,\mu\right).
\end{equation}

However, the previous estimator relies on the unobserved manifold $\mathcal{M}$ and its underlying geodesic distance $\delta$. A popular estimator is given by the Isomap algorithm for $\delta$.  The idea is to build a simple metric graph from the data, which will be close enough from the manifold. Hence the approximation of the geodesic distance between two points  depends on the length of the edges of the graph which connect these points. The algorithm approximates the unknown geodesic distance $\delta$ between all pairs of points in $\mathcal{M}$ in terms of shortest path distance between all pairs of points in a nearest neighbor graph $\mathcal{G}$ constructed from the data points $\mathcal{E}$. If the discretization of the manifold contains enough points with regards to the curvature of the manifold, hence the graph distance will be a good approximation of the geodesic distance. For details about the Isomap algorithm, see \citet{Tenenbaum-2000}, \citet{BdSLT-00}, and \citet{Silva2003}.

The construction of the weighted neighborhood graph in the first step of the Isomap algorithm requires the choice of a parameter which controls the neighborhood size and therefore its success. This is made according to a $K-$rule (connecting each point with its $K$ nearest neighbors) or $\epsilon-$rule (connecting each point with all points lying within a ball of radius $\epsilon$) which are closely related to the local curvature of the manifold. Points which are  too distant to be  connected to the biggest graph are not used, making the algorithm unstable (see \citet{Balasubramanian2002}). In this paper we propose a robust version of this algorithm which leads to an approximation of the geodesic distance, $\hat{\delta}$. Our version does not exclude any point and does not require any additional tuning parameter.  This algorithm has been applied with success to align density curves in microarray data analysis (task known as normalization in bioinformatics) in \citet{Gallon-Loubes-Maza-13}. The construction of the approximated geodesic distance is detailed in Section~\ref{section2}.

Once an estimator of the geodesic distance is built, we propose to estimate the empirical Fr\'echet median by its approximated version
\begin{equation}\label{intrinsicStatisticempirical}
\widehat{\mu}_{I,n}^1=\argmin_{\mu\in\mathcal{G}}\sum_{i=1}^n \hat{\delta}\left(X_i,\mu\right).
\end{equation}

This estimator is restricted to stay within the graph $\mathcal{G}$ since the approximated geodesic distance is only defined on the graph. Hence we choose as a pattern of the observation the point which is at the {\it center} of the dataset, where center has to be understood with respect to the inner geometry of the observations.
\section{The robust manifold learning algorithm}
\label{section2}

Let $X$ be a random variable with values in an unknown connected and geodesically complete Riemannian manifold $\mathcal{M}\subset\mathbb{R}^m$, and a sample $\mathcal{E}=\{X_i\in\mathcal{M},\:i=1,\dots,n\}$ with distribution $Q$. Set $\mathrm{d}$ the Euclidean distance on $\mathbb{R}^m$ and $\delta$ the induced geodesic distance on $\mathcal{M}$. Our aim is to estimate  the geodesic distance between two points on the manifold $\delta\left(X_{i},X_{i'}\right)$ for all $i\neq i'\in \left\lbrace 1,\ldots,n \right\rbrace$.

The Isomap algorithm proposes to learn the manifold topology from a neighborhood graph. In the same way, our purpose is to approximate the geodesic distance $\delta$ between a pair of data points by the graph distance on the shortest path between the pair on the neighborhood graph. The main difference between our algorithm and the Isomap algorithm lies in the treatment of points which are far from the others. Indeed, the first step of the original Isomap algorithm consists in constructing the $K$-nearest neighbor graph or the $\epsilon$-nearest neighbor graph for a given positive integer $K$ or a real $\epsilon>0$, respectively and then to exclude  points which are not connected to the  graph. Such a step is not present in our algorithm since we consider that a distant point is not always considered an outlier. Hence, we do not exclude any points.   Moreover, a sensitive issue of the Isomap algorithm is that it requires the choice of the neighbor parameter ($K$ or $\epsilon$) which is closely related to the local curvature of the manifold, determining the quality of the construction (see, for instance, \citet{Balasubramanian2002}). In our algorithm, we give a tuning parameter free process to simplify the analysis.

The algorithm has three steps. The first step constructs a complete weighted graph associated to $\mathcal{E}$ based on Euclidean distances $\text{d}(X_{i},X_{i'})$ between all pairwise points $X_{i},X_{i'}\in\mathbb{R}^m$. It is a complete Euclidean graph $\mathcal{G}_{\text{E}}=\left(\mathcal{E},E\right)$ with set of nodes $\mathcal{E}$ and set of edges $E=\left\{\left\{X_{i},X_{i'}\right\},\:i=1,\dots,n-1,\:i'=i+1,\dots,n\right\}$ weighted with the corresponding Euclidean distances.

In the second step, the algorithm obtains the Euclidean Minimum Spanning Tree $\mathcal{G}_{\text{MST}}=\left(\mathcal{E},E_\text{T}\right)$ associated to $\mathcal{G}_{\text{E}}$, i.e. the spanning tree that minimizes the sum of the weights of the edges in the spanning tree of $\mathcal{G}_{\text{E}}$, $\sum_{\{X_{i},X_{i'}\}\in E_\text{T}}\mathrm{d}\left(X_{i},X_{i'}\right)$. The underlying idea in this construction is that, if two points $X_{i}$ and $X_{i'}$ are relatively close, then we have that $\delta\left(X_{i},X_{i'}\right)\approx\mathrm{d}\left(X_{i},X_{i'}\right)$. This may not be true if the manifold is very twisted and/or if too few points are observed, and may induce bad approximations. So the algorithm will produce a good approximation for relatively regular manifolds. This drawback is well known when dealing with graph-based approximations of the geodesic distance (\citet{Tenenbaum-2000}, and \citet{Silva2003}).

An approximation of $\delta\left(X_{i},X_{i'}\right)$ is provided by the sum of all the Euclidean distances of the edges of the shortest path on $\mathcal{G}_{\text{MST}}$ connecting $X_{i}$ to $X_{i'}$, i.e. $\hat{\delta}\left(X_{i},X_{i'}\right)=\min_{g_{ii'}\in \mathcal{G}_{\text{MST}}}L\left(g_{ii'}\right)$, where $L\left(g_{ii'}\right)$ denotes the length of a path $g_{ii'}$ connecting $X_{i}$ to $X_{i'}$ on $\mathcal{G}_\text{MST}$. However, this construction is highly unstable since the addition of new points may change completely the structure of the graph.

To cope with this problem, we propose in the third stage to add more robustness in the construction of the approximation graph.  Actually, in our algorithm  we add more edges between the data points to add extra paths and thus to cover better the manifold. The underlying idea is that paths which are close to the ones selected in the construction of the $\mathcal{G}_\text{MST}$ could also provide good alternate ways of connecting the edges. Closeness here is understood  as lying in open balls $B\left(X_i,\epsilon_i\right)\subset\mathbb{R}^m$ around the point $X_i$ with radius $\epsilon_i=\max_{\left\{X_{i},X_{i}\right\}\in E_\text{T}}\mathrm{d}\left(X_{i},X_{i'}\right)$. Hence, these new paths between the data are admissible and should be added to the edges of the graph. Finally, we obtain a new robustified graph $\mathcal{G}^\prime=\left(\mathcal{E},E^\prime\right)$ defined by
\[
\{X_{i},X_{i'}\}\in E^\prime\Longleftrightarrow\overline{X_{i}X_{i'}}\subset\bigcup_{i=1}^nB\left(X_i,\epsilon_i\right),
\]
where
\[
\overline{X_{i}X_{i'}}=\left\{X\in\mathbb{R}^m,\:\exists\lambda\in[0,1],\:X=\lambda X_{i}+(1-\lambda)X_{i'}\right\}.
\]

\noindent Finally, $\mathcal{G}^\prime$ is the graph which gives rise to our estimator of $\delta$, given by
\begin{equation}\label{estidelta}
\hat{\delta}\left(X_{i},X_{i'}\right)=\min_{g_{ii'}\in \mathcal{G}^\prime}L\left(g_{ii'}\right).
\end{equation}

Hence, $\hat{\delta}$ is the distance associated with $\mathcal{G}^\prime$, that is, for each pair of points $X_{i}$ and $X_{i'}$, we have $\hat{\delta}\left(X_{i},X_{i'}\right)=\mathrm{L}\left(\hat{\gamma}_{ii'}\right)$ where $\hat{\gamma}_{i}$ is the minimum length path between $X_{i}$ and $X_{i'}$ associated to $\mathcal{G}^\prime$. We point out that all points of the data sets are connected in the new graph $\mathcal{G}^\prime$.

\noindent A summary of the procedure is gathered in the Algorithm~\ref{algorithm}

\begin{algorithm}[!ht]
\caption{Robust approximation of $\delta$}
\label{algorithm}
\begin{algorithmic}[1] 
\REQUIRE $\mathcal{E}=\left\{X_{i}\in\mathbb{R}^{m},\: i=1,\ldots,n\right\}$
\ENSURE  $\hat{\delta}$
\STATE	 Calculate $\text{d}(X_{i},X_{i'})=\|X_{i}-X_{i'}\|_{2}$ between all pairwise data points $X_{i}$ and $X_{i'}$, $i=1,\dots,n-1,\:i'=i+1,\dots,n$, and construct the complete Euclidean graph $\mathcal{G}_\text{E}=\left(\mathcal{E},E\right)$ with set of edges $E=\left\{\left\{X_{i},X_{i'}\right\}\right\}$.
\STATE	 Obtain the Euclidean Minimum Spanning Tree $\mathcal{G}_\text{MST}=\left(\mathcal{E},E_\text{T}\right)$ associated to $\mathcal{G}_\text{E}$.
\STATE	 For each $i=1,\ldots,n$ calculate $\epsilon_i=\max_{\left\{X_{i},X_{i'}\right\}\in E_\text{T}}\mathrm{d}\left(X_{i},X_{i'}\right)$, and open balls $B\left(X_i,\epsilon_i\right)\subset\mathbb{R}^m$ of center $X_i$ and radius $\epsilon_i$. Construct a graph $\mathcal{G}^\prime=\left(\mathcal{E},E^\prime\right)$ adding more edges between points according to the rule
\[
\{X_{i},X_{i'}\}\in E^\prime\Longleftrightarrow\overline{X_{i}X_{i'}}\subset\bigcup_{i=1}^nB\left(X_i,\epsilon_i\right),
\]
where $\overline{X_{i}X_{i'}}=\left\{X\in\mathbb{R}^m,\:\exists\lambda\in[0,1],\:X=\lambda X_{i}+(1-\lambda)X_{i'}\right\}$.
\STATE Estimate the geodesic distance between two points by the length of the shortest path in the graph $\mathcal{G}^\prime$ between these points using the Floyd's or Dijkstra's algorithm (see, e.g. \citet{Lee-Verleysen-07}).
\end{algorithmic}
\end{algorithm}

%\begin{algorithm}
%\caption{Robust approximation of $\delta$}
%\label{algorithm}
%\begin{algorithmic}[1] 
%\REQUIRE $\mathcal{E}=\left\{X_{i}\in\mathbb{R}^{m},\: i=1,\ldots,n\right\}$
%\ENSURE  $\hat{\delta}$
%\FOR	 {$i=1,\dots,n-1$}
%\FOR 	 {$i'=i+1,\dots,n$}
%\STATE	 $\text{d}(X_{i},X_{i'})=\|X_{i}-X_{i'}\|_{2}$
%\ENDFOR
%\ENDFOR
%\STATE	 Construct a complete Euclidean graph $\mathcal{G}_{E}=\left(\mathcal{E},E\right)$ with set of edges $E=\left\{\left\{X_{i},X_{i'}\right\}\right\}$.
%\STATE	 Obtain the Euclidean Minimum Spanning Tree $\mathcal{G}_{MST}=\left(\mathcal{E},E_{T}\right)$ associated to $\mathcal{G}_{E}$.
%\FOR {$i=1,\ldots,n$} 
%\STATE $\epsilon_i=\max_{\left\{X_{i},X_{i'}\right\}\in E_{T}}\mathrm{d}\left(X_{i},X_{i'}\right)$
%\STATE $B\left(X_i,\epsilon_i\right)\subset\mathbb{R}^m$
%\ENDFOR
%\STATE Construct a graph $\mathcal{G}^\prime=\left(\mathcal{E},E^\prime\right)$ adding more edges between points according to the rule
%\[
%\{X_{i},X_{i'}\}\in E^\prime\Longleftrightarrow\overline{X_{i}X_{i'}}\subset\bigcup_{i=1}^nB\left(X_i,\epsilon_i\right),
%\]
%where $\overline{X_{i}X_{i'}}=\left\{X\in\mathbb{R}^m,\:\exists\lambda\in[0,1],\:X=\lambda X_{i}+(1-\lambda)X_{i'}\right\}$.
%\STATE Compute the shortest path distances $\hat{\delta}\left(X_{i},X_{i'}\right)$ between all pairs of points $X_{i}$ and $X_{i'}$ in the $\mathcal{G}^\prime$ using Floyd's or Dijkstra's algorithm.
%\end{algorithmic}
%\end{algorithm}

Note that, the 3-step algorithm described above contains widespread graph-based methods to achieve our purpose. In this article, all graph-based calculations, such as Minimum Spanning Tree estimation or shortest path calculus, were carried out  with the \texttt{igraph} package for network analysis by \citet{Csardi2006}.

An illustration of the algorithm and its behavior when the number of observations increases are displayed respectively in Figures~\ref{met1} and~\ref{met2}. In Figure~\ref{met1}, points $\left(X_i^1,X_i^2\right)_i$ are simulated as follows:
\begin{equation}\label{sim1}
X_i^1=\frac{2i-n-1}{n-1}+\epsilon_i^1,\quad \text{ and }\quad X_i^2=2\left(\frac{2i-n-1}{n-1}\right)^2+\epsilon_i^2,
\end{equation}
where $\epsilon_i^1$ and $\epsilon_i^2$ are independent and normally distributed with mean 0 and variance 0.01 for $i=1,\ldots,n$ and $n=30$. In Figure~\ref{met2}, some results of graph $\mathcal{G}^\prime$ for $n=10,30,100$ are given. We can see that graph $\mathcal{G}^\prime$ tends to be close to the true manifold $\left\{\left(t,t^2\right)\in\mathbb{R}^2,\:t\in\mathbb{R}\right\}$ when $n$ increases.

Obviously, this estimation shows that the recovered structures in Figures~\ref{met1} and~\ref{met2} are pretty sensitive to noise. Nevertheless, to estimate a representative of a sample of curves, a prior smoothing step is almost always carried out  as in \citet{Ramsay-Silverman-05}. This is done  in Section~\ref{section5} for our real data sets.

\begin{figure}[!htb]
	\centering
	\includegraphics[width=0.333\textwidth,angle=270]{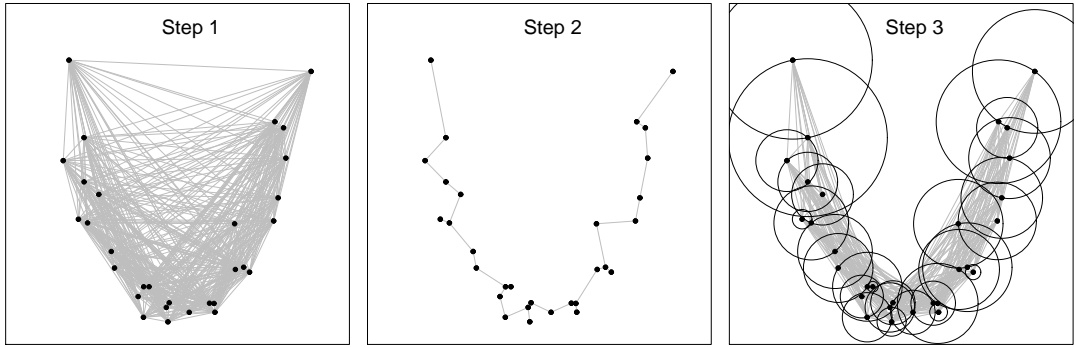}
	\caption[3-step construction of a subgraph $\mathcal{G}^\prime$ from Simulation~\eqref{sim1}]{The 3-step construction of a subgraph $\mathcal{G}^\prime$ from Simulation~\eqref{sim1}. On the left, the simulated data set (black dots) and the associated complete Euclidean graph $\mathcal{G}_\text{E}$ (Step 1). On the middle, the $\mathcal{G}_\text{MST}$ associated with the complete graph $\mathcal{G}_\text{E}$ (Step 2). On the right, the associated open balls and the corresponding subgraph $\mathcal{G}^\prime$ (Step 3).}
	\label{met1}
\end{figure}
\begin{figure}[htb]
	\centering
	\includegraphics[width=0.333\textwidth,angle=270]{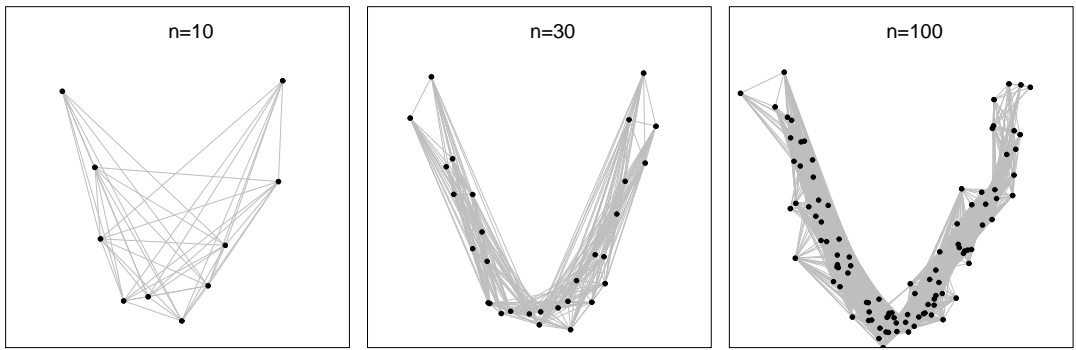}
	\caption{Evolution of graph $\mathcal{G}^\prime$ from Simulation~\eqref{sim1} for $n=10,30,100$.}
	\label{met2}
\end{figure}

\section{Application: Template estimation in a shape invariant model}
\label{section3}

In this section, we consider the case where the observations are curves warped from an unknown template $f\colon\mathcal{T}\to\mathbb{R}$. We want to study whether the {\it central} curve defined previously as the median of the data with respect to the geodesic distance provides a good pattern of the curves. Good means, in that particular case, that the intrinsic median should be close to the pattern $f$.

We consider a translation model indexed by a real valued random variable $A$ with unknown distribution on an interval $(b,c)\subset\mathbb{R}$
\begin{equation}\label{model_shift}
X_{ij} = f_{i}(t_{j}) = f\left(t_j - A_i\right),\:i = 1,\ldots,n,\:j = 1,\ldots,m,
\end{equation}
where $\left(A_i\right)_i$ are i.i.d random variables drawn with distribution $A$ which models the unknown shift parameters. This specification is an special case of the self-modeling regression mentioned in the introduction.

Under a nonparametric registration model, \citet{Maza2006} and \citet{Dupuy-Loubes-Maza-11} define the \textit{structural expectation} function of a sample of curves and build a registration procedure in order to estimate it. Following the same philosophy, but for the case of the translation effect model~\eqref{model_shift}, we propose to use as a good pattern of the dataset the \textit{Structural Median} function $f_\mathrm{SM}$ defined as 
%Under a nonparametric registration model, \citet{Maza2006} and \citet{Dupuy-Loubes-Maza-11}  propose to use as a good pattern of the dataset  the so-called \textit{structural expectation}  $f_\mathrm{SM}$ defined as 
\begin{equation}\label{def_structuralMedian}
f_\mathrm{SM} = f\left(\cdot-\mathrm{med}(A)\right),
\end{equation}
where $\mathrm{med}(A)$ denotes the median of $A$. % They build a registration procedure in order to estimate  $f_\mathrm{SM}$.

We will see that the manifold embedding point of view enables to recover this pattern. Actually,  define a one-dimensional  function in $\mathcal{M}\subset\mathbb{R}^m$ parameterized by a  parameter $a\in(b,c)\subset\mathbb{R}$ as
\[
\begin{split}
X\colon(b,c)&\to\mathbb{R}^m\\
	   a\;\;&\mapsto X(a)=\left(f\left(t_1-a\right),\ldots,f\left(t_m-a\right)\right),
\end{split}
\]
and set $\mathcal{C}=\left\{X(a)\in\mathbb{R}^m,\:a\in(b,c)\right\}$.

As soon as $X$ is a regular curve, that is, if its first derivative never vanishes,
\begin{equation}\label{cond}
X'\neq0\Longleftrightarrow\forall a\in(b,c),\:\exists j\in\{1,\dots,m\},\:f'\left(t_j-a\right)\neq0,
\end{equation}
then, the smooth mapping $X\colon a\mapsto X(a)$ provides a natural parametrization of $\mathcal{C}$ which can thus be seen as a submanifold of $\mathbb{R}^m$ of dimension 1 (\citet{Carmo1992}). The corresponding geodesic distance is given by
\begin{equation}\label{geodC}
\delta\left(X(a_1),X(a_2)\right)=\left|\int_{a_1}^{a_2}\left\|X^\prime(a)\right\|\mathrm{d}a\right|,
\end{equation}
with $X'(a)=\text{d}X(a)/\text{d}a=\left(\text{d}X_{1}(a)/\text{d}a,\ldots,\text{d}X_{m}(a)/\text{d}a\right)^\top$.

The observation model~\eqref{model_shift} can then be seen as a discretization of the manifold $\mathcal{C}$ for different values $\left(A_i\right)_i$. Hence, finding the intrinsic median of all shifted curves can be done by understanding the geometry of space $\mathcal{C}$, and thus, by approximating the geodesic distance between observed curves.
Define the intrinsic median with respect to the geodesic distance~\eqref{geodC} on $\mathcal{C}$, that is
\begin{equation}\label{to}
\widehat{\mu}_I^1=\argmin_{\mu\in\mathcal{C}}\sum_{i=1}^n\delta\left(X_i,\mu\right).
\end{equation}
The following theorem gives a minimizer.

\begin{theorem}\label{thmSM}
Under the assumption~\eqref{cond} that $X$ is a regular curve, we get that
\[
\widehat{\mu}_I^1 = \left(f\left(t_1-\widehat{\mathrm{med}}(A)\right),\ldots,f\left(t_m-\widehat{\mathrm{med}}(A)\right)\right), \]
where $\widehat{\mathrm{med}}(A)$ is the empirical median.
\end{theorem}

Hence as soon as we observe a sufficient number of curves to ensure that the median and the empirical median are close, the intrinsic median is a natural approximation of \eqref{def_structuralMedian}.  Therefore, the manifold framework provides a geometrical interpretation of the structural median of a sample of curves. The estimator is thus given by 
\begin{equation}\label{estimator}
\widehat{\mu}_{I,n}^1=\argmin_{\mu\in\mathcal{E}}\sum_{i=1}^n\widehat{\delta}\left(X_i,\mu\right),
\end{equation}
where $\widehat{\delta}$ is an approximation of the unknown underlying geodesic distance, that is estimated by the algorithm described in Section~\ref{section2}.

We point out that in many situations, giving a particular model for the deformations corresponds actually to consider a particular manifold embedding for the data. Once the manifold is known, its corresponding  geodesic distance may be properly computed, as done in the translation case. So in some particular cases, the minimization in~\eqref{to} can give an explicit formulation and then it is possible to identify the resulting Fr\'echet median. Hence previous theorem may be generalized to such cases as done in \citet{Gallon-Loubes-Maza-13}.

Note first that this case only holds for the Fr\'echet median ($\alpha =1$) but not the mean for which the so-called structural expectation and the Fr\'echet mean are different. Moreover, the choice of  the median is also driven by the need for a robust method, whose good behavior will be highlighted in the simulations and applications in the following sections.

As  shown in the simulation study below, when the observations can be modeled by a set of curves warped  from an unknown template  by a general deformation  process, estimate~\eqref{estimator} enables to recover the main pattern in a better way than classical methods. Obviously, the method relies on the assumption that all the observed data belong to an embedded manifold whose geodesic distance can be well approximated by the proposed algorithm.

\section{Simulation study}
\label{section4}

In this section, the numerical properties of our estimator, called Robust Manifold Embedding (RME), defined by the equation~\eqref{estimator} in Section~\ref{section3} are studied using simulated data. The estimator is compared to those obtained with the Isomap algorithm and the Modified Band Median (MBM) estimator proposed by \citet{ArribasGil-Romo-12}, which is based on the concept of depth for functional data (see \citet{Lopez-Romo-09}). %We also compare the estimators with the cross-sectional median.
The behavior of the estimator when the number of curves is increased is also analyzed.

Four different types of simulations of increasing warping complexity for the single shape invariant model were carried out, observing $n=15,30,45,60$ curves on $m=100$ equally spaced discrete points $\left(t_j\right)_j$ in the interval $[-10,10]$. The experiment was conducted with $R=100$ repetitions. The template function $f$ and shape invariant model, for each simulation, are given as follows:

\noindent\textbf{\textit{Simulation 1}:} One-dimensional manifold defined by $f(t)=5\sin(t)/t$ and 
\[
X_{ij}=f\left(t_j+A_i\right),
\]
where $\left(A_i\right)_i$ are i.i.d uniform random variables on interval $[-5,5]$.

\noindent\textbf{\textit{Simulation 2}:} Two-dimensional manifold given by $f(t)=5\sin(t)$ and
\[
X_{ij}=f\left(A_it_j+B_i\right),
\]
where $\left(A_i\right)_i$ and $\left(B_i\right)_i$ are independent and (respectively) i.i.d uniform random variables on intervals $[0.7,1.3]$ and $[-1,1]$.

\noindent\textbf{\textit{Simulation 3}:} Four-dimensional manifold given by $f(t)=t\sin(t)$ and
\[
X_{ij}=A_if\left(B_it_j+C_i\right)+D_i,
\]
where $\left(A_i\right)_i$, $\left(B_i\right)_i$, $\left(C_i\right)_i$ and $\left(D_i\right)_i$ are independent and (respectively) i.i.d uniform random variables on intervals $[0.7,1.3]$, $[0.7,1.3]$, $[-1,1]$ and $[-1,1]$.

\noindent\textbf{\textit{Simulation 4}:} Four-dimensional manifold given by $f(t)=\phi t + t \sin(t)\cos(t)$ with $\phi=0.9$, and
\[
X_{ij}=A_if\left(B_it_j+C_i\right)+D_i,
\]
where $\left(A_i\right)_i$, $\left(B_i\right)_i$, $\left(C_i\right)_i$ and $\left(D_i\right)_i$ as in the Simulation 3.

Figure~\ref{fig03} illustrates the simulated data sets from Simulations 1-4 with $n=30$ curves for one simulation run (one of 100 repetitions). % For all of simulations it is clear that the cross-sectional median underperforms all other methods.
For Simulation 1, where there is only phase variability, all methods %, except the median,
follows the structural characteristics of the sample of curves, where the template estimated by the robust manifold approach is the closest curve to the theoretical function. The same conclusion can be inferred from Simulation 2. Indeed, for this simulation type, and for this particular simulation run, the RME estimator coincides with the theoretical template function. For the four-dimensional manifolds in Simulations 3 and 4, where there is an additional amplitude variability, the robust manifold estimator captures better the structural pattern in the sample of curves followed by the Isomap estimator. Note that in the Simulation 4, both approaches coincide. Although the MBM estimator follows the shape of the theoretical template, the estimator deviates from it in the cases 2-4.

%For simulations three and four, where there is additional amplitude variability, the best graphical results are achieved by the Isomap and robust manifold estimators. Note that in the simulation four, both of approaches coincide. On the contrary, the MBM estimator has an opposite pattern with respect to the theoretical template.

\begin{figure}[!ht]
	\centering
	\includegraphics[height=0.49\textwidth,width=0.44\textwidth,angle=270]{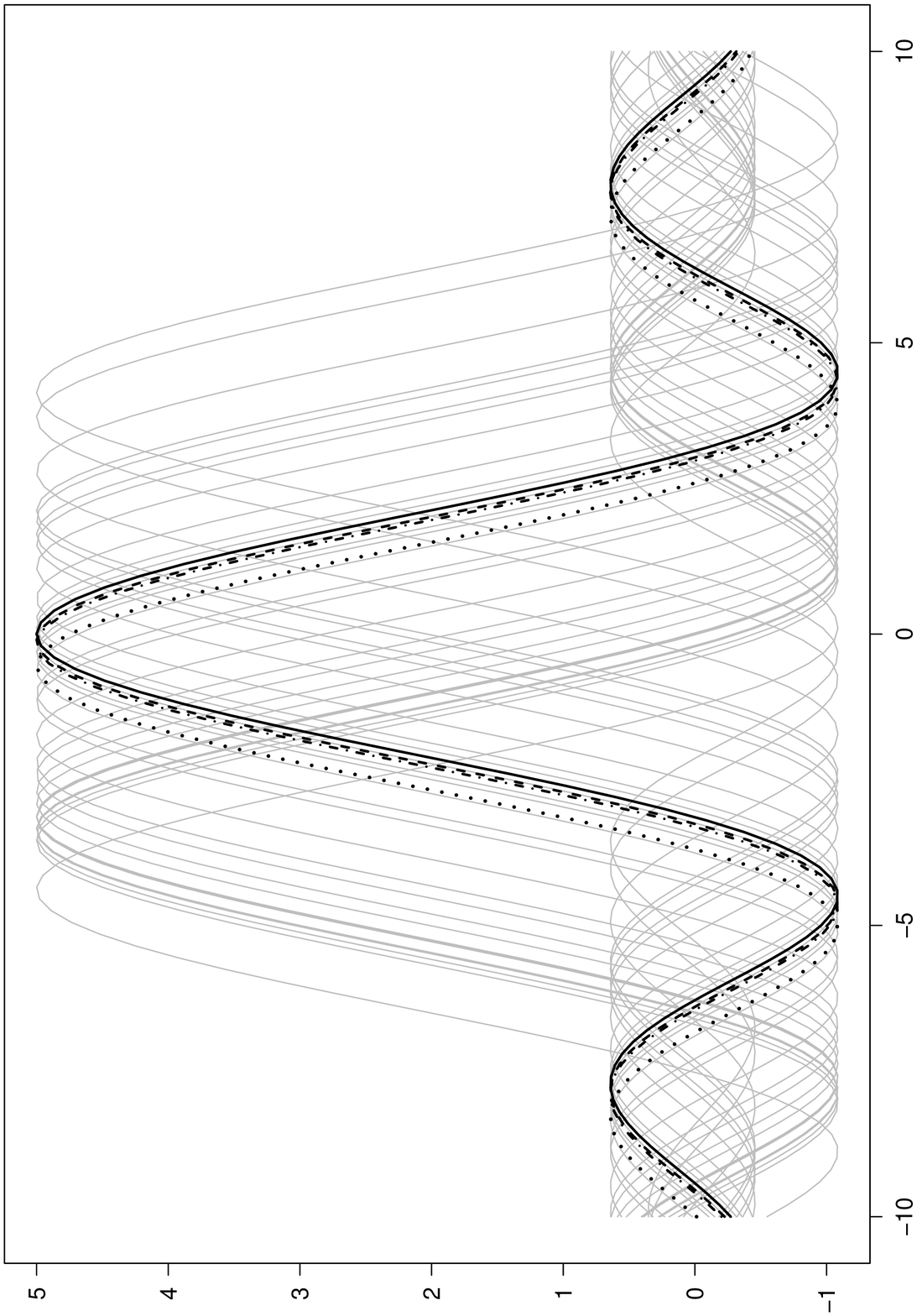}
	\includegraphics[height=0.49\textwidth,width=0.44\textwidth,angle=270]{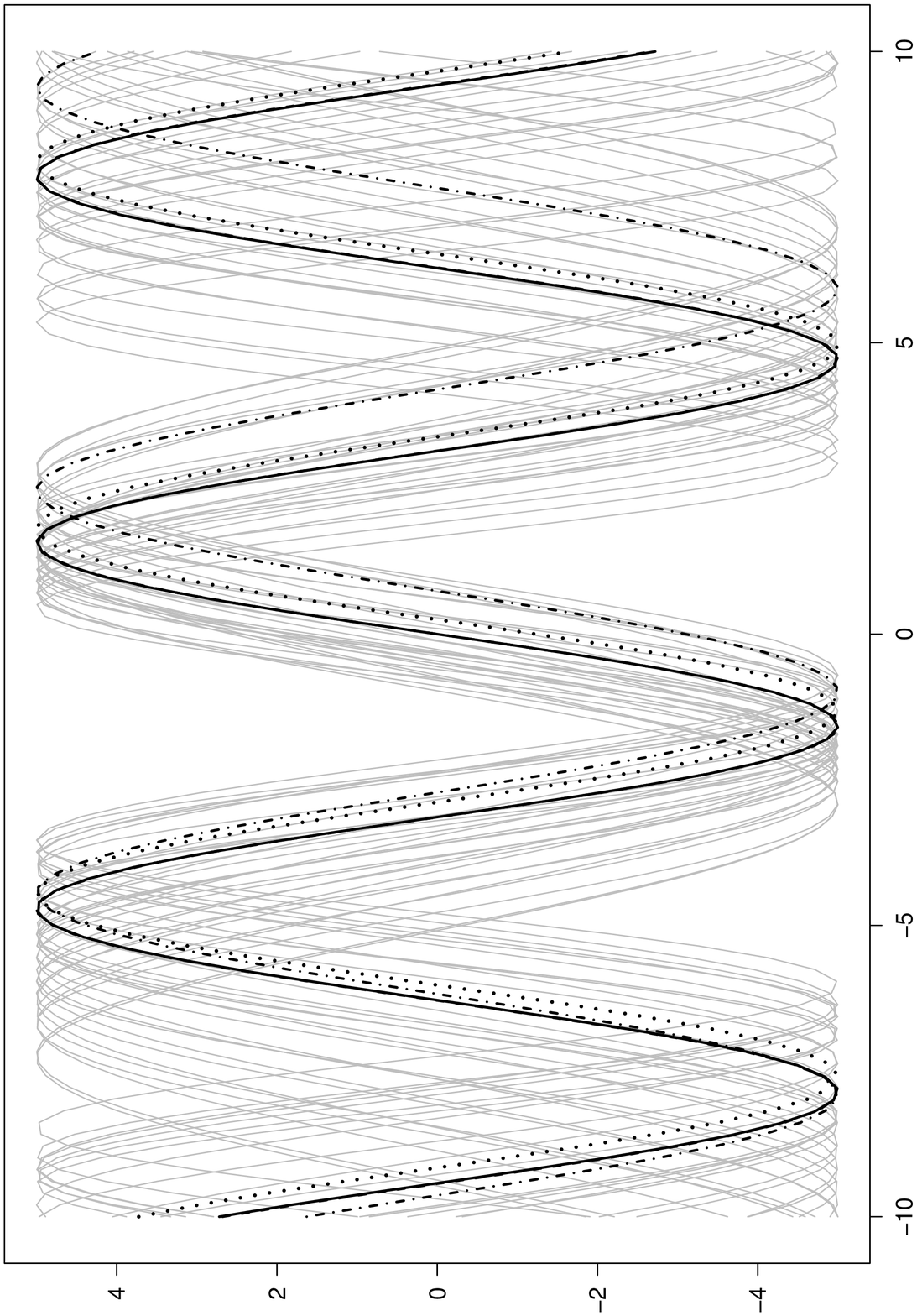}
	\includegraphics[height=0.493\textwidth,width=0.44\textwidth,angle=270]{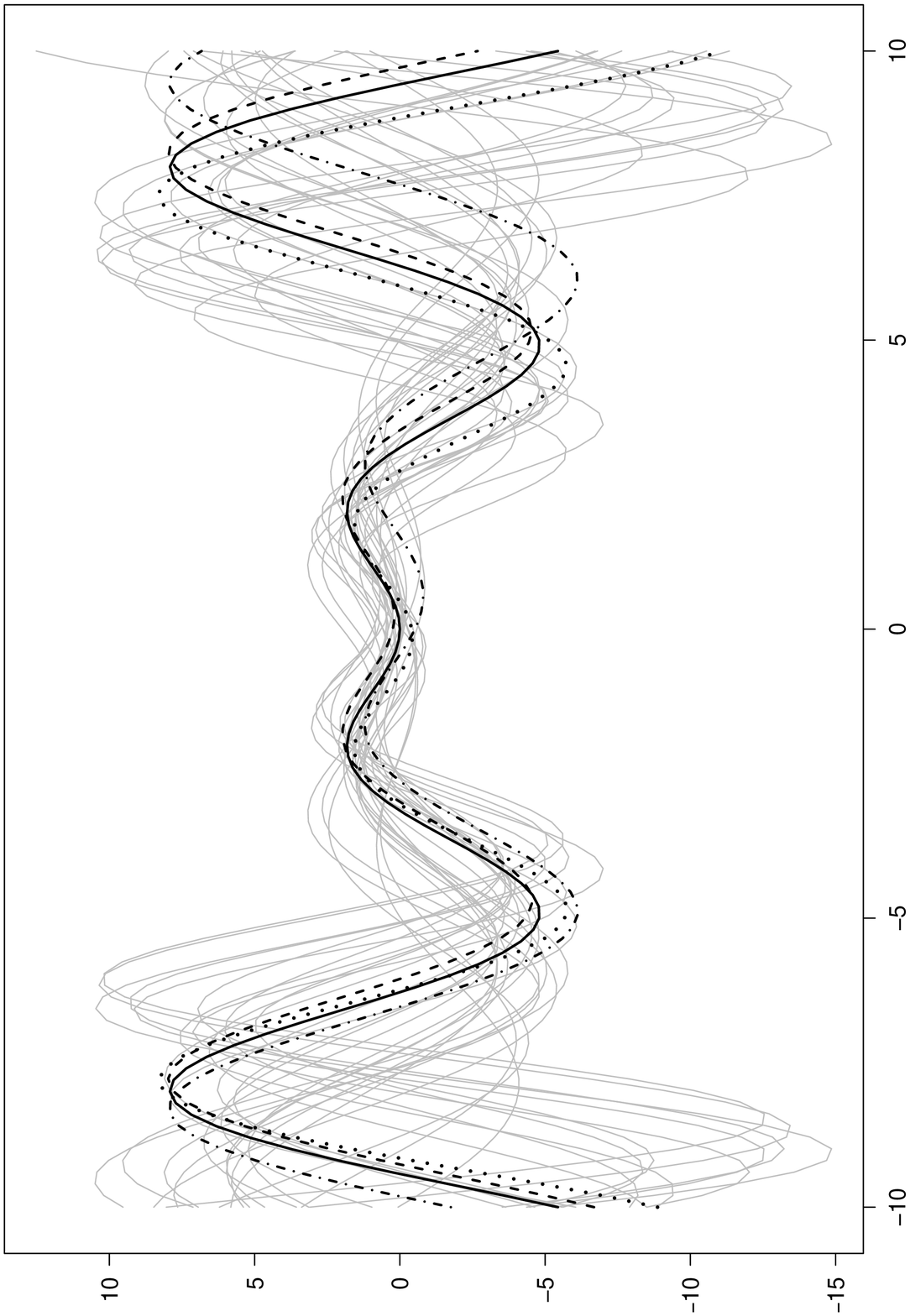}
	\includegraphics[height=0.493\textwidth,width=0.44\textwidth,angle=270]{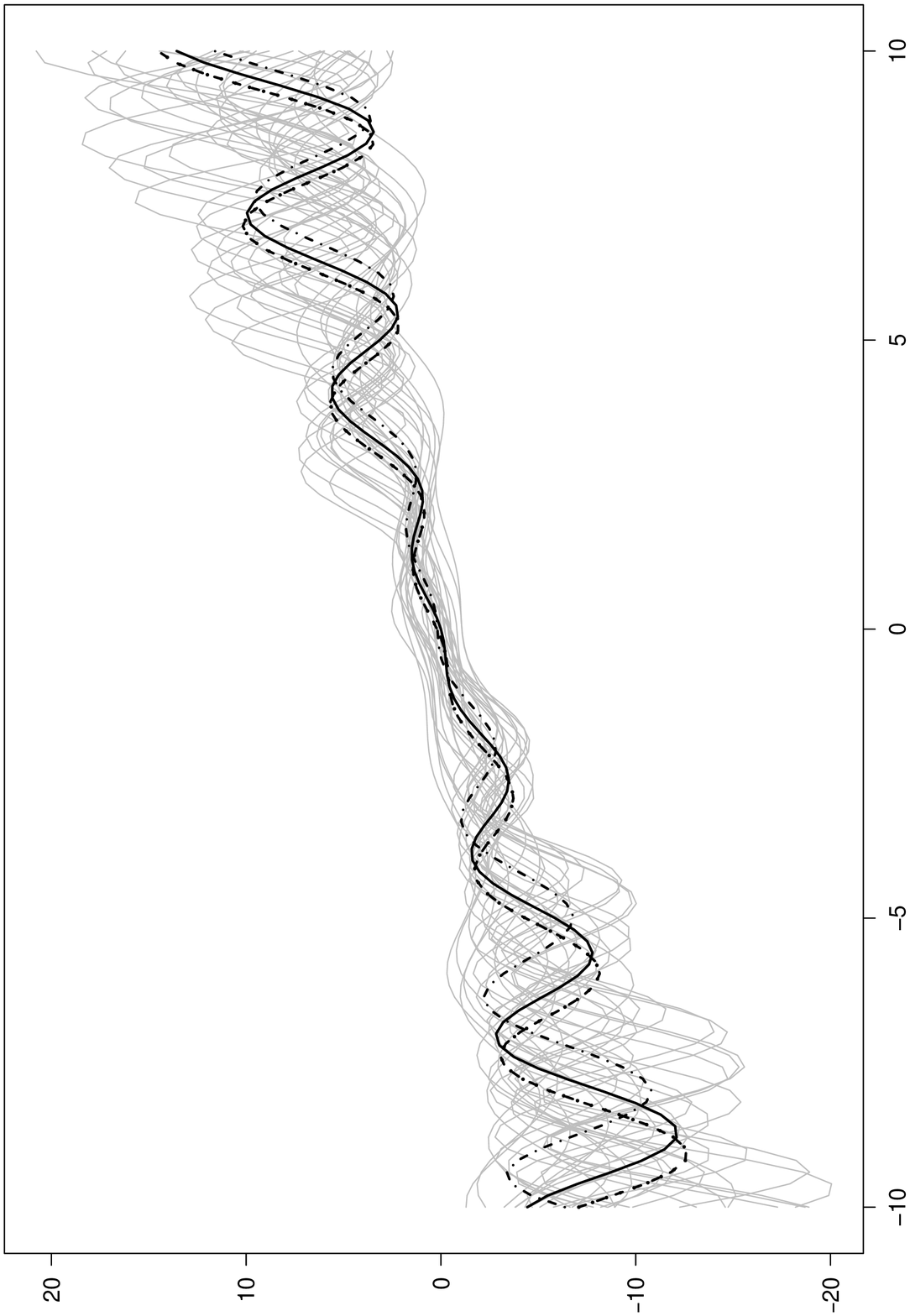}
	\caption[Simulated curves from Simulation 1-4]{Simulated curves (gray) from Simulation 1 (top left), 2 (top right), 3 (bottom left) and 4 (bottom right) for one simulation run, and the respective target template function $f$ (solid line), MBM (dash-dotted line), Isomap (dotted line), and RME (dashed line) template estimators.}
	\label{fig03}
\end{figure}

In order to compare more accurately the estimators described above, we calculate, for each one, the empirical mean squared error obtained on the $R=100$ repetitions of each type of simulation. We recall the definition, for estimator $\hat{f}$ of a given type of simulation, of the mean squared error:
\[
\text{Mean Squared Error}\left(\hat{f}\right)=\frac{1}{R}\sum_{r=1}^R\|\hat{f}_r-f\|_2^2,
\]
where, $\hat{f}_r$ is the estimation from the $r$-th repetition of the given simulation type, $f$ is the true template function and $\|\cdot\|_2$ is the classical Euclidean norm. We also highlight, for our comparisons, the fact that
\[
\text{Mean Squared Error}\left(\hat{f}\right)=\underbrace{\frac{1}{R}\sum_{r=1}^R\|\hat{f}_r-\bar{f}\|_2^2}_{\mbox{Variance}}+\underbrace{\|\bar{f}-f\|_2^2}_{\text{Squared bias}},
\]
where $\bar{f}$ is the mean of all $R$ obtained estimations.

Table~\ref{tab:comp} shows the mean squared errors, variances and squared biases of each estimator for simulations 1, 2, 3 and 4, and for different number on curves $n=15,30,45,60$ in the sample. Values have been rounded to zero decimal places to facilitate the comparisons, and the minimum values are signed in bold.

From the table, we can observe that when the number of curves in the sample is small ($n=15$) the MBM estimator has better results in terms of the MSE, Bias2 and variance for the Simulation 1. The same is true when $n=30$. With $n=45,60$ curves the MBM estimator has minimum mean squared error and variance, and our estimator has smaller bias. Comparing the Isomap and RME methods only, the latter overcomes the former. For Simulation 2, the RME estimator overcomes the MBM and Isomap estimators for $n=30,45,60$ curves, except when $n=15$, where the Isomap estimator is better. However, in this case there are not big differences between Isomap and RME estimators. As we expected, when the geometry of the curves is more complex, i.e. when we have a four-dimensional manifolds, the results are more variated. For Simulation 3, the RME estimator has a good performance with $n=45,60$. With $n=15,30$ the better results are shared by the MBM and Isomap methods. In Simulation 4, the MBM estimator has, in general, better results. Finally, note that although the theorem in Section~\ref{section3} is valid for one-dimensional manifolds generated by time shifts (Simulation 1), we can see that the intrinsic sample median estimator by approximating the corresponding geodesic distance with the robust algorithm performs well for manifolds of higher dimension (Simulations 2-4).

\begin{table}[!ht]
\caption{Comparison of estimators for Simulations 1-4 with different sample sizes.}
\label{tab:comp}
\centering{\footnotesize{
\begin{tabular}{clcccccc}\\\hline
\multirow{2}{*}{$\boldsymbol{n}$} & \multirow{2}{*}{\textbf{Statistic}} & \multicolumn{3}{c}{\textbf{Simulation1}} & \multicolumn{3}{c}{\textbf{Simulation 2}}\\\cmidrule{3-8}
	&			 &	MBM		&	Isomap	&	RME	&		MBM	 &	Isomap	&	RME	\\\hline
	&	MSE		 &	\textbf{136}		&	389	&	335		&		790	&	\textbf{400}	&	435	\\
15	&	Bias2	 &	\textbf{23}		&	152	&	118		&		141	&	\textbf{35}	&	46	\\
	&	Variance &	\textbf{113}		&	236	&	217		&		649	&	\textbf{366}	&	389	\\\hline
	&	MSE		 &	\textbf{30}		&	108	&	92		&		666	&	338	&	\textbf{268}	\\
30	&	Bias2	 &	\textbf{8}		&	13	&	10		&		98	&	34	&	\textbf{18}	\\
	&	Variance &	\textbf{22}		&	95	&	82		&		568	&	304	&	\textbf{249}	\\\hline
	&	MSE		 &	\textbf{24}		&	139	&	66		&		669	&	244	&	\textbf{155}	\\
45	&	Bias2	 &	10		&	23	&	\textbf{5}		&		120	&	20	&	\textbf{9}	\\
	&	Variance &	\textbf{14}		&	116	&	60		&		549	&	224	&	\textbf{147}	\\\hline
	&	MSE		 &	\textbf{14}		&	85	&	55		&		634	&	168	&	\textbf{136}	\\
60	&	Bias2	 &	5		&	9	&	\textbf{4}		&		161	&	5	&	\textbf{4}	\\
	&	Variance &	\textbf{10}		&	76	&	51		&		472	&	163	&	\textbf{132}	\\\hline
\multirow{2}{*}{$\boldsymbol{n}$} & \multirow{2}{*}{\textbf{Statistic}} & \multicolumn{3}{c}{\textbf{Simulation 3}} & \multicolumn{3}{c}{\textbf{Simulation 4}}\\\cmidrule{3-8}
	&			 &	MBM	&	Isomap	&	RME	&	MBM	 &	Isomap	&	RME	\\\hline
	&	MSE		 &	1350	&	\textbf{1152}	&	1171	&		\textbf{876}	&	890	&	893	\\
15	&	Bias2	 &	\textbf{251}	&	375	&	441	&		\textbf{394}	&	512	&	522	\\
	&	Variance &	1098	&	777	&	\textbf{730}	&	483	&	378	&	\textbf{371}	\\\hline
	&	MSE		 &	1025	&	\textbf{673}	&	721	&	911	&	\textbf{861}	&	876	\\
30	&	Bias2	 &	212	&	\textbf{160}	&	248	&		\textbf{470}	&	536	&	554	\\
	&	Variance &	813	&	513	&	\textbf{473}	&		441	&	325	&	\textbf{323}	\\\hline
	&	MSE		 &	1034	&	553	&	\textbf{498}	&	\textbf{820}	&	827	&	868	\\
45	&	Bias2	 &	223	&	\textbf{105}	&	141	&		\textbf{397}	&	524	&	569	\\
	&	Variance &	811	&	449	&	\textbf{356}	&		423	&	303	&	\textbf{299}	\\\hline
	&	MSE		 &	965	&	572	&	\textbf{402}	&		879	&	\textbf{776}	&	842	\\
60	&	Bias2	 &	168	&	\textbf{97}	&	122	&		\textbf{458}	&	474	&	563	\\
	&	Variance &	797	&	475	&	\textbf{280}	&		421	&	302	&	\textbf{279}	\\\hline
\end{tabular}}
}
\end{table}

\subsection*{Robustness analysis}

In order to assess the robustness of the RME estimator, we carried out an additional simulation study generating atypical curves in the functional data sets. Atypical curves can be considered in the two following settings: either real outliers which do not share any common shape with the observations or curves that are obtained by atypical deformations within the same model. The methodology that we propose is based on the existence of a common shape between the data, hence adding complete outliers may break the geometry of the data, i.e. the manifold structure, if the outliers are far from any deformations. This remark is true for any registration method when curves away from the warping model are considered. Hence robust methods without the constraint of a deformation model, for instance based on the notion of depth function, behave better. Yet, we show in the following that when considering curves warped by a deformation process which has some outliers, our method is robust with respect to the Isomap procedure and may compete with the MBM estimator.

Thereby, from $n=15,30,45,60$ curves we generated 10\% of them as atypical according to the single shape invariant model in the four type of simulations considered above, modifying the corresponding individual shift parameters but preserving the geometric structure of the curves. So, for each simulation, the non-atypical curves $X_{ij}$ with $i=1,\ldots,(n-\lceil 0.10n\rceil)$ are generated as above, and the atypical curves $\tilde{X}_{ij}$ with $i=(n-\lceil 0.10n\rceil)+1,\ldots,n$ were generated as:

\noindent\textbf{\textit{Simulation 1}:} One-dimensional manifold defined by $f(t)=5\sin(t)/t$ and 
\[
\tilde{X}_{ij}=f\left(t_j+\tilde{A}_i\right),
\]
where $\big(\tilde{A}_i\big)_i$ are i.i.d uniform random variables on interval $[4.5,6]$.

\noindent\textbf{\textit{Simulation 2}:} Two-dimensional manifold given by $f(t)=5\sin(t)$ and
\[
\tilde{X}_{ij}=f\left(\tilde{A}_it_j+\tilde{B}_i\right),
\]
where $\big(\tilde{A}_i\big)_i$ and $\big(\tilde{B}_i\big)_i$ are independent and (respectively) i.i.d uniform random variables on intervals $[0.35,0.65]$ and $[-0.5,0.5]$.

\noindent\textbf{\textit{Simulation 3}:} Four-dimensional manifold given by $f(t)=t\sin(t)$ and
\[
\tilde{X}_{ij}=\tilde{A}_if\left(\tilde{B}_it_j+\tilde{C}_i\right)+\tilde{D}_i,
\]
where $\big(\tilde{A}_i\big)_i$, $\big(\tilde{B}_i\big)_i$, $\big(\tilde{C}_i\big)_i$ and $\big(\tilde{D}_i\big)_i$ are independent and (respectively) i.i.d uniform random variables on intervals $[1.3,1.4]$, $[0.7,1.3]$, $[-1.5,-1]$ and $[1,1.5]$.

\noindent\textbf{\textit{Simulation 4}:} Four-dimensional manifold given by $f(t)=\phi t + t \sin(t)\cos(t)$ with $\phi=0.9$, and
\[
\tilde{X}_{ij}=\tilde{A}_if\left(\tilde{B}_it_j+\tilde{C}_i\right)+\tilde{D}_i,
\]
where $\big(\tilde{A}_i\big)_i$, $\big(\tilde{B}_i\big)_i$, $\big(\tilde{C}_i\big)_i$ and $\big(\tilde{D}_i\big)_i$ are independent and (respectively) i.i.d uniform random variables on intervals $[1.05,1.95]$, $[1.05,1.95]$, $[-1,1]$ and $[-1,1]$.

Figure~\ref{fig03b} illustrates the simulated data sets from Simulations 1-4 with $n=30$ curves for one replication. The curves signed in red color correspond to the atypical curves. For this particular simulation run, we see that the atypical curves has influence over the Isomap estimator for all types of simulation. For the one and two-dimensional manifolds in Simulations 1 and 2 respectively we observe that the RME estimator has a good performance. For example, note, as in the simulation study without atypical curves developed above, the RME estimator coincides with the theoretical template function for the simulation type 1, and for this particular simulation run. For complex shape functions as in Simulations 3 and 4, our estimator captures adequately the common pattern of the sample in presence on atypical curves. As expected, the depth-based estimator is robust against atypical observations.

\begin{figure}[!ht]
	\centering
	\includegraphics[height=0.49\textwidth,width=0.44\textwidth,angle=270]{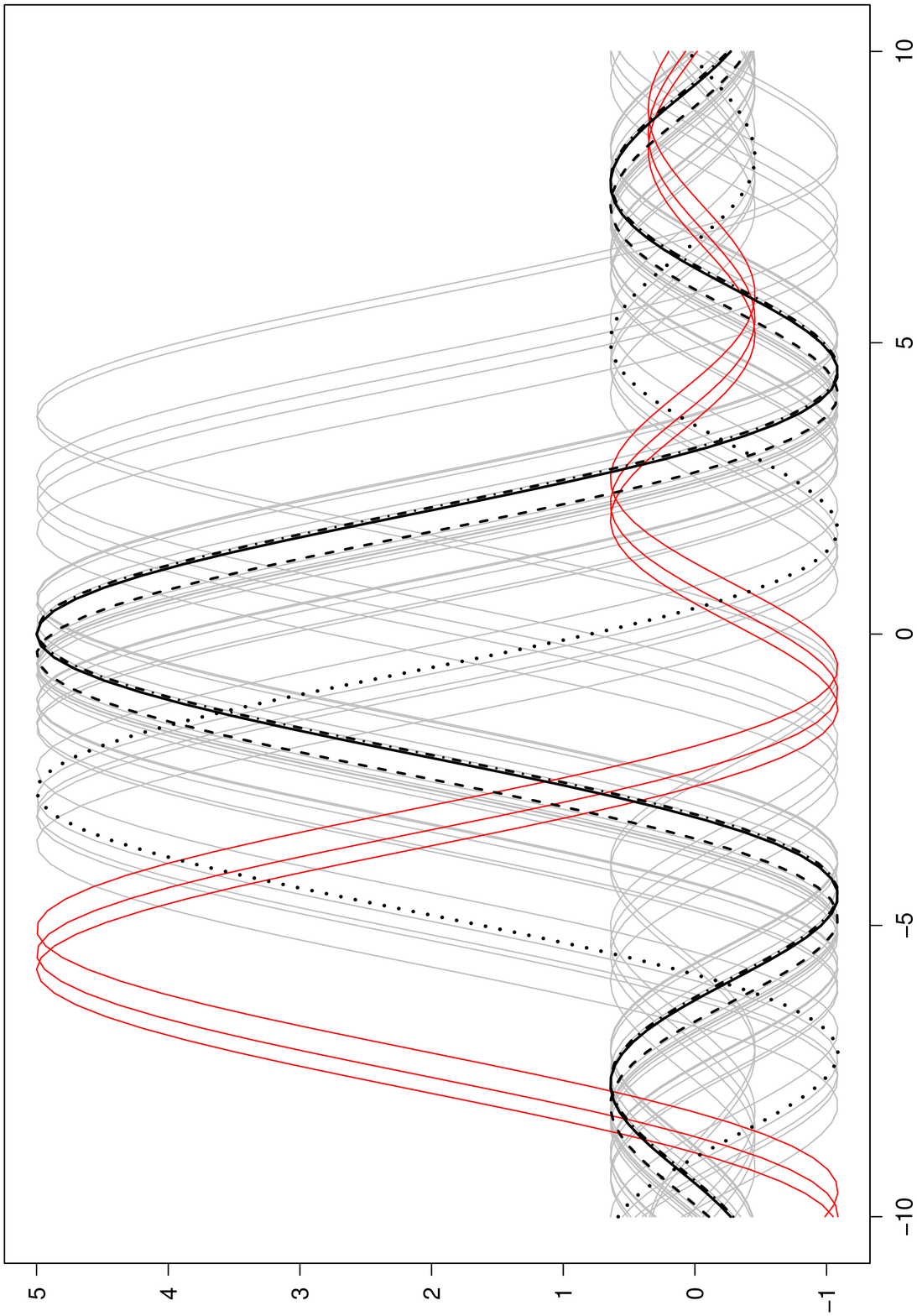}
	\includegraphics[height=0.49\textwidth,width=0.44\textwidth,angle=270]{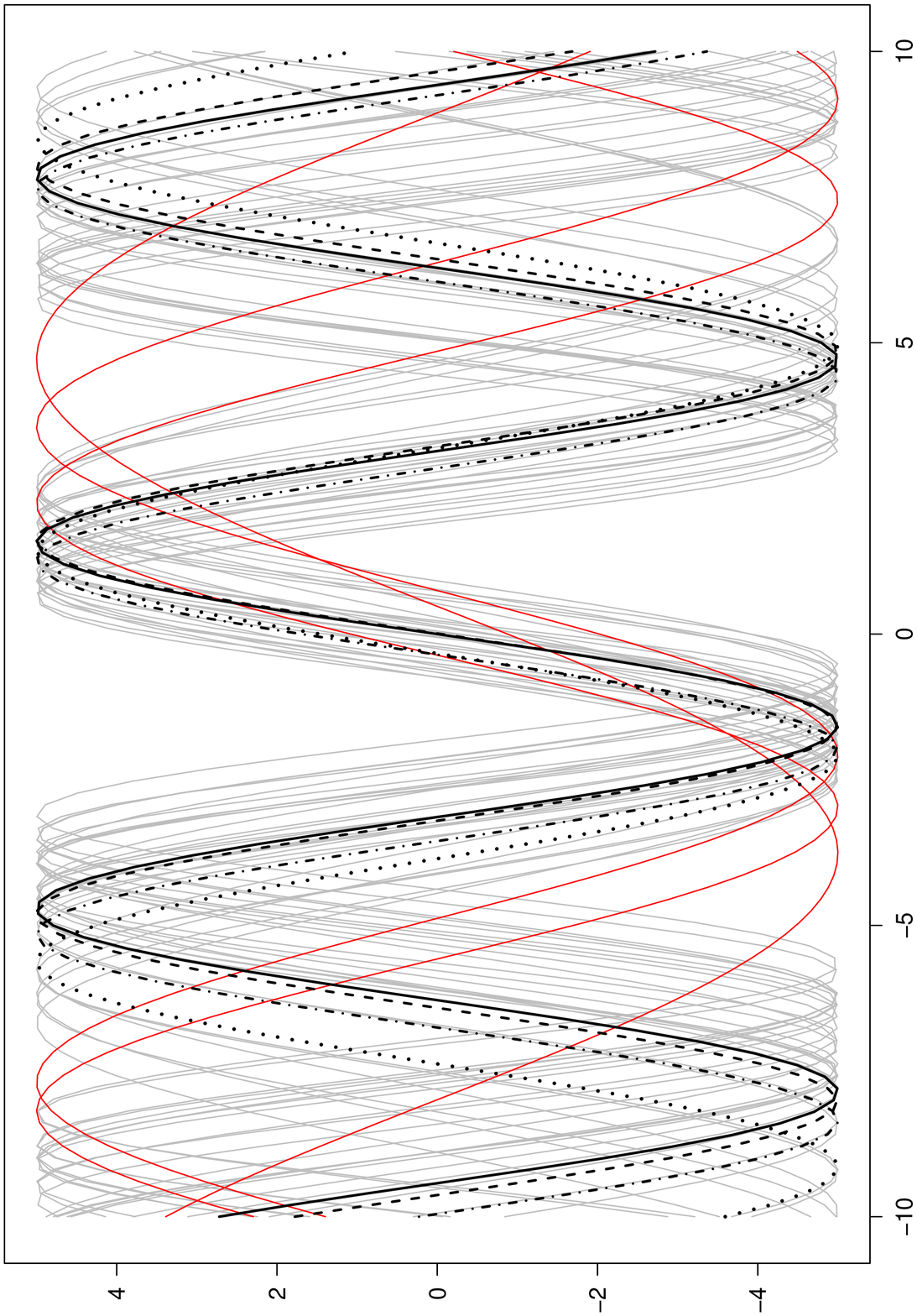}
	\includegraphics[height=0.493\textwidth,width=0.44\textwidth,angle=270]{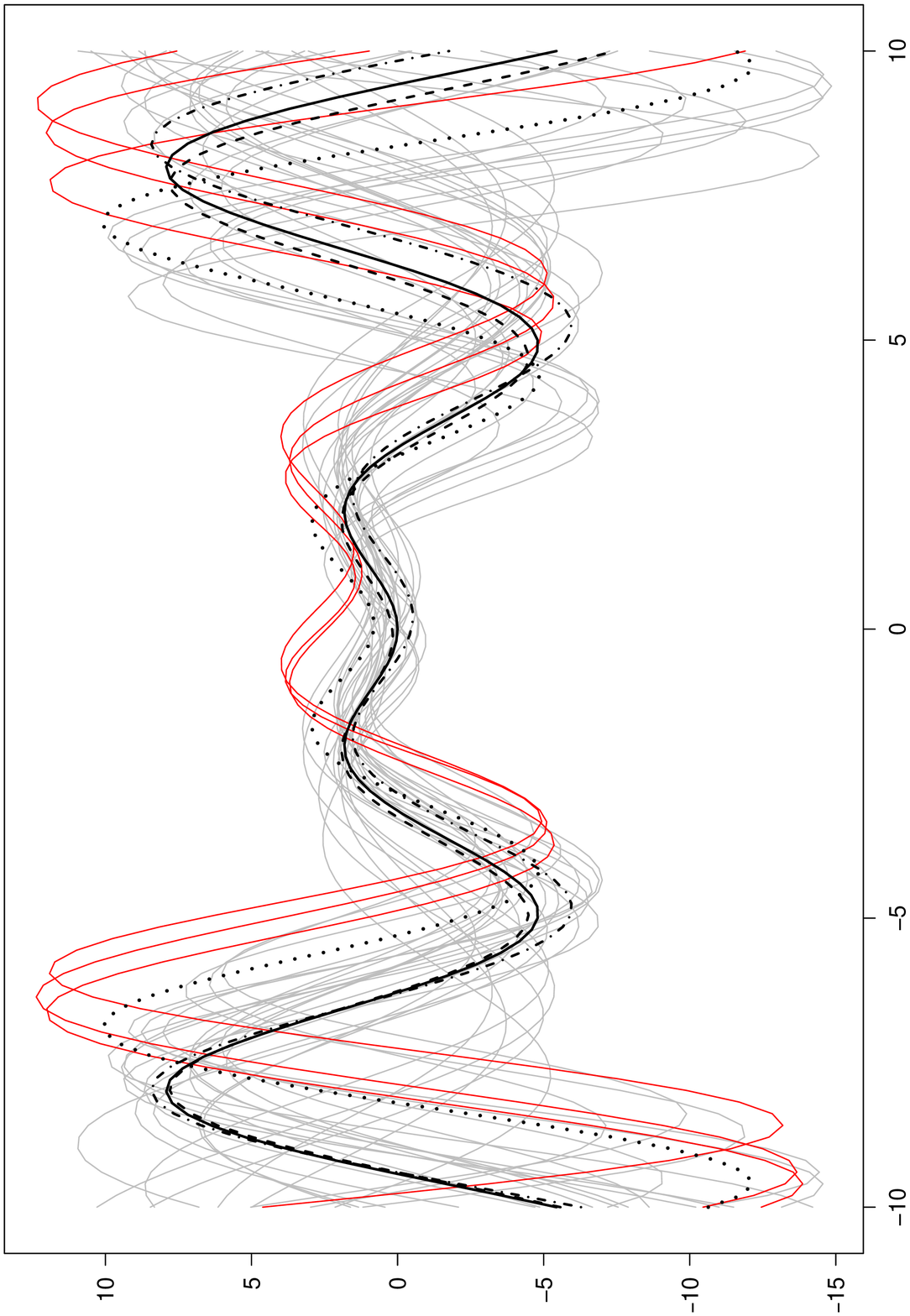}
	\includegraphics[height=0.493\textwidth,width=0.44\textwidth,angle=270]{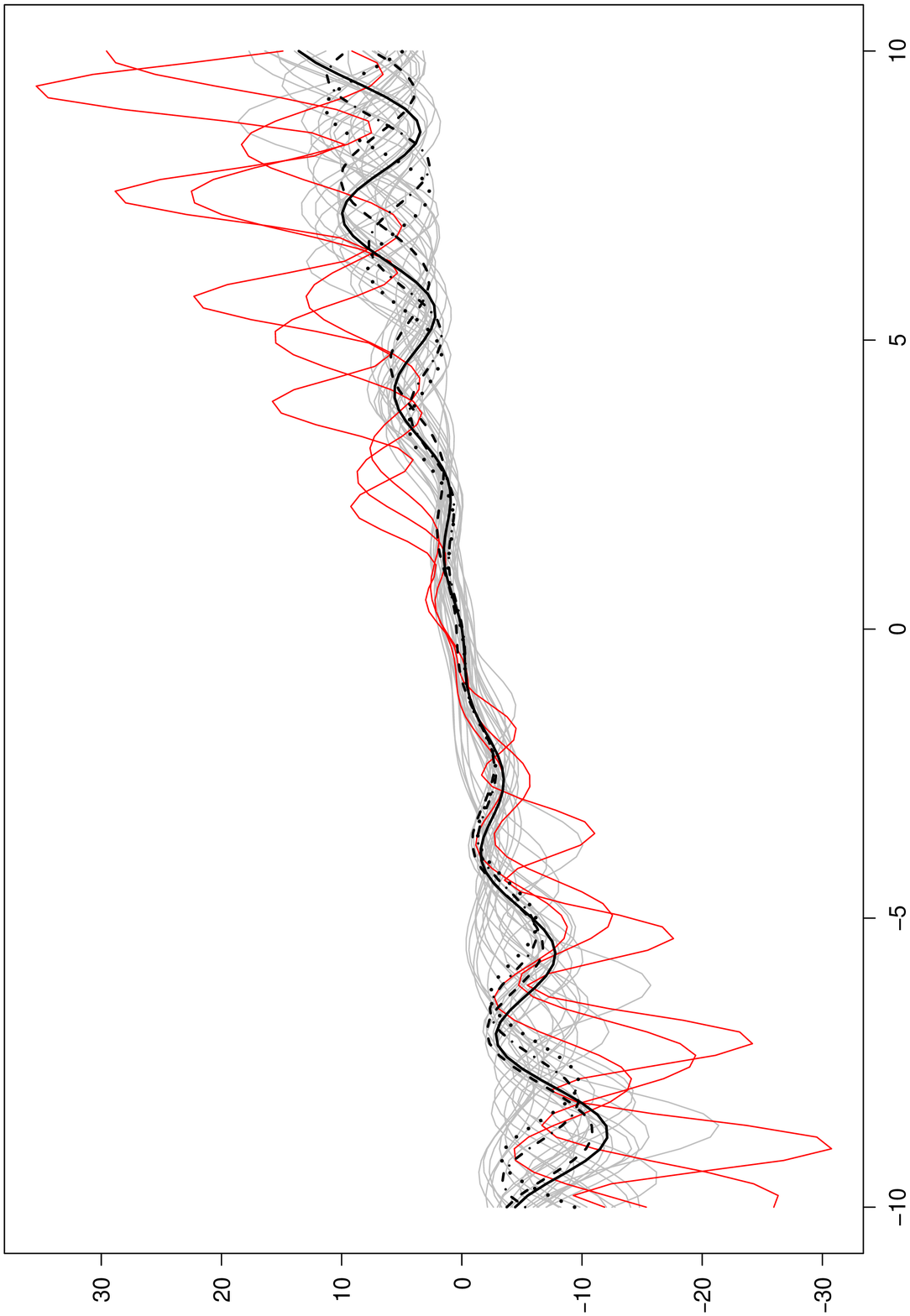}
	\caption[Simulated curves from simulation 1, 2, 3 and 4]{Simulated curves (gray) from simulation type 1 (top left), 2 (top right), 3 (bottom left) and 4 (bottom right) including atypical curves (red) for one simulation run, and the respective target template function $f$ (solid line), MBM (dash-dotted line), Isomap (dotted line), and RME (dashed line) template estimators.}
	\label{fig03b}
\end{figure}

The mean squared errors, variances and squared biases of each estimator for Simulations 1-4 and different number on curves $n=15,30,45,60$ including atypical curves are showed in the Table~\ref{tab:comp2}. For Simulation 1, the MBM estimator has the best results for all number of curves. In this case, the RME method overcomes its not robust version estimator (Isomap). Additionally, when the warping complexity increases, the RME estimator has minimum mean squared errors in most cases for Simulations 2-4. As expected, only the when the number of curves is small ($n=15$) the estimator performs less well.

\begin{table}[!ht]
\caption{Comparison of estimators including atypical curves for different sample sizes.}
\label{tab:comp2}
\centering{\footnotesize{
\begin{tabular}{clcccccc}\\\hline
\multirow{2}{*}{$\boldsymbol{n}$} & \multirow{2}{*}{\textbf{Statistic}} & \multicolumn{3}{c}{\textbf{Simulation1}} & \multicolumn{3}{c}{\textbf{Simulation 2}}\\\cmidrule{3-8}
	&			 &		MBM	&	Isomap	&	RME	&		MBM	 &	Isomap	&	RME	\\\hline
	&	MSE	&	\textbf{107}	&	452	&	400	&	830	&	\textbf{570}	&	680	\\	
15	&	Bias2	&	\textbf{21}	&	215	&	169	&	186	&	\textbf{76}	&	116	\\	
	&	Variance	&	\textbf{85}	&	237	&	232	&	645	&	\textbf{495}	&	564	\\	\hline
	&	MSE	&	\textbf{36}	&	177	&	166	&	649	&	409	&	\textbf{300}	\\	
30	&	Bias2	&	\textbf{8}	&	49	&	46	&	115	&	39	&	\textbf{20}	\\	
	&	Variance	&	\textbf{28}	&	128	&	120	&	535	&	370	&	\textbf{280}	\\	\hline
	&	MSE	&	\textbf{21}	&	121	&	81	&	523	&	296	&	\textbf{212}	\\	
45	&	Bias2	&	\textbf{9}	&	32	&	26	&	89	&	23	&	\textbf{11}	\\	
	&	Variance	&	\textbf{13}	&	89	&	56	&	433	&	273	&	\textbf{200}	\\	\hline
	&	MSE	&	\textbf{19}	&	151	&	90	&	551	&	276	&	\textbf{212}	\\	
60	&	Bias2	&	\textbf{9}	&	45	&	40	&	98	&	63	&	\textbf{46}	\\	
	&	Variance	&	\textbf{10}	&	106	&	51	&	453	&	213	&	\textbf{166}	\\	\hline
\multirow{2}{*}{$\boldsymbol{n}$} & \multirow{2}{*}{\textbf{Statistic}} & \multicolumn{3}{c}{\textbf{Simulation 3}} & \multicolumn{3}{c}{\textbf{Simulation 4}}\\\cmidrule{3-8}
	&			 &		MBM	&	Isomap	&	RME	&	MBM	 &	Isomap	&	RME	\\\hline
	&	MSE	&	1387	&	\textbf{1093}	&	1098	&	990	&	983	&	\textbf{963}	\\	
15	&	Bias2	&	\textbf{257}	&	327	&	396	&	\textbf{485}	&	561	&	538	\\	
	&	Variance	&	1129	&	766	&	\textbf{702}	&	505	&	\textbf{422}	&	425	\\	\hline
	&	MSE	&	1370	&	\textbf{856}	&	857	&	901	&	861	&	\textbf{856}	\\	
30	&	Bias2	&	260	&	\textbf{204}	&	312	&	\textbf{434}	&	505	&	511	\\	
	&	Variance	&	1110	&	652	&	\textbf{545}	&	467	&	355	&	\textbf{345}	\\	\hline
	&	MSE	&	1206	&	640	&	\textbf{547}	&	874	&	863	&	\textbf{860}	\\	
45	&	Bias2	&	234	&	\textbf{155}	&	197	&	556	&	541	&	\textbf{406}	\\	
	&	Variance	&	972	&	484	&	\textbf{350}	&	\textbf{317}	&	322	&	454	\\	\hline
	&	MSE	&	963	&	585	&	\textbf{462}	&	924	&	864	&	\textbf{861}	\\	
60	&	Bias2	&	154	&	\textbf{118}	&	165	&	\textbf{500}	&	537	&	561	\\	
	&	Variance	&	809	&	468	&	\textbf{297}	&	424	&	327	&	\textbf{301}	\\	\hline
\end{tabular}}
}
\end{table}

\section{Applications}
\label{section5}

In this section we apply the proposed robust manifold learning algorithm to extract the template function of a sample of curves on three real datasets of functional data: the well-known Berkeley Growth and Gait data in functional data applications (\citet{Ramsay-Silverman-05}), and a reflectance data of two landscape types. Our algorithm is compared with the Isomap and Modified Band Median methods.

\subsection{Berkeley growth study}

The data of the Berkeley's study consist in 31 height measurements for 54 girls and 38 boys recorded between the ages of 1 and 18 years. Intervals between measurements range from 3 months (age 1-2 years), to yearly (age 3-8), to half-yearly (age 8-18). One of the goals with this kind of data is the pattern analysis of growth velocity and acceleration curves, represented by the first and second derivatives of the height functions, in order to characterize its spurts and trends during years. The velocity and acceleration curves for girls and boys were obtained by taking the first and second order differences, respectively, of the height curves, whose functional representations were made using a B-spline smoothing (see \citet{Ramsay-Silverman-05} for details).

Figure~\ref{growth} provides the smoothed velocity curves (on the top) measured in centimeters per year (cm/year) and the smoothed acceleration curves (on the bottom) measured in centimeters per squared year (cm/year$^2$) of height for girls (on the left) and for boys (on the right). It is evident that all individuals exhibit a common velocity and acceleration pattern throughout years, but features as peaks, troughs and inflection points occur at different times for each child.

From all of the graphs in the Figure~\ref{growth}, we see, in general, that all the template estimators obtain a curve situated in the middle of the samples of curves capturing its common shape pattern appropriately. For the case of sample velocity curves of girls (top-left graph) the RME and MBM estimators coincide. The Isomap estimator deviates slightly from the center. In the case of samples of velocity and acceleration curves of boys, both the RME and Isomap estimators choose the same template function. Only in the case of acceleration curves of girls (bottom-left graph), the three methods choose different functions. In summary, we infer that the RME estimator seems to perform a good work extracting a meaningful shape curve.

%In Figure~\ref{growth}, we see that for the case of sample acceleration curves of boys (bottom-right graph), the template of all methods coincide, selecting the same template curve capturing the common shape pattern appropriately. In the case of acceleration curves of girls (bottom-left graph), the three methods choose different functions. Nevertheless, the Isomap and Robust Manifold estimators have a similar behavior and MBM estimator is quite different. Note also that, although the conceptual idea of the MBM estimator is to search for a central function in the sample, it seems to choose curves that departs from the center in some time intervals, like in the velocity (in both girls and boys cases) and acceleration (in the case of girls) curves. For the velocity curves of girls, the Isomap and MBM methods choose the same curve, and in the case of boys our robust estimator and the Isomap procedure coincide. In general, we see that the RME seems to play a good work extracting a meaningful shape curve coinciding with at least one of the two other methods.

\begin{figure}[!ht]
	\centering
	%height=0.68\textheight
	\includegraphics[height=\textwidth,width=\textwidth,angle=270]{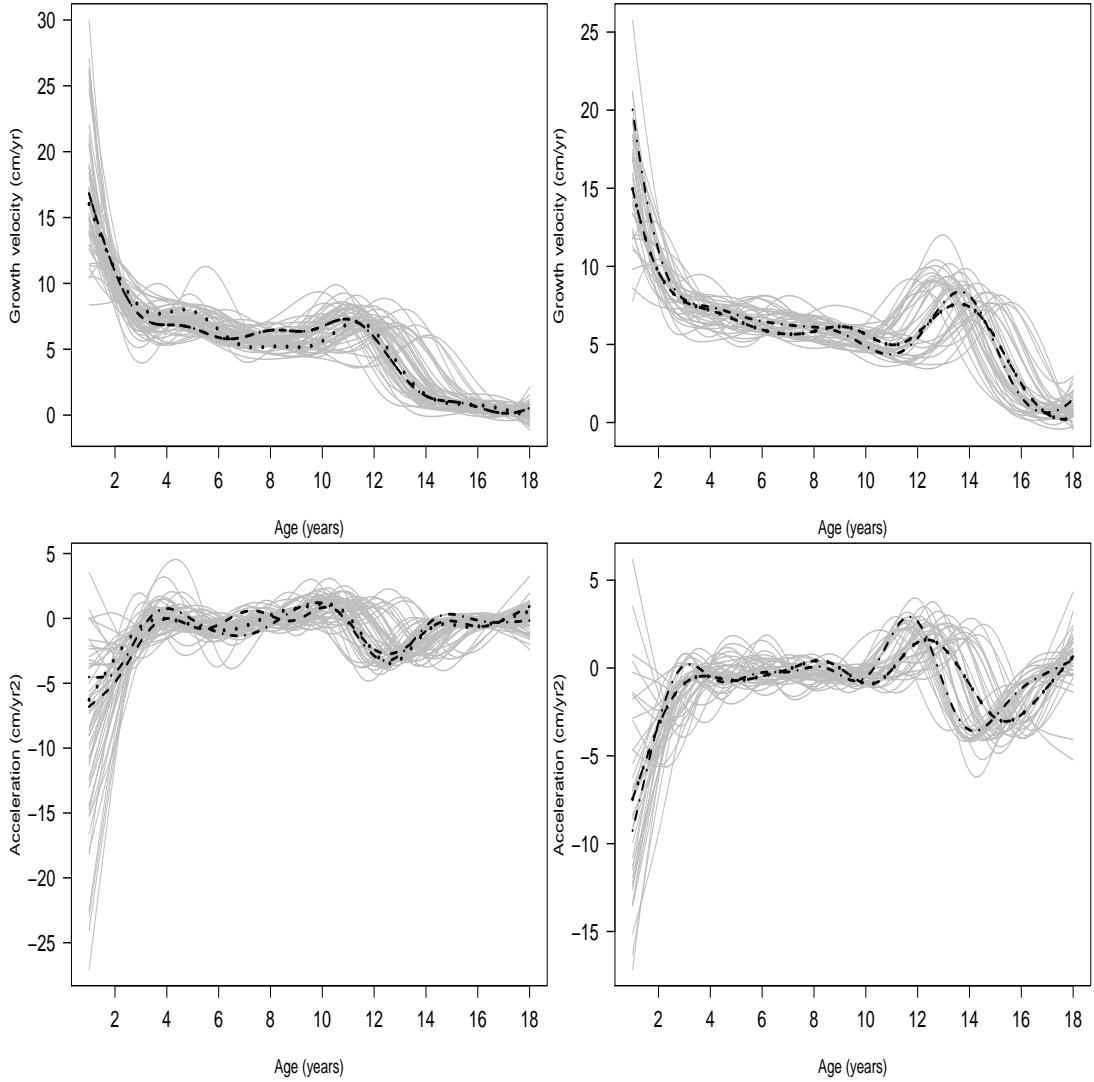}
	\caption[]{Velocity (on the top) and acceleration (on the bottom) curves of 54 girls (on the left) and 31 boys (on the right) in the Berkeley growth study (gray lines). The estimated template functions with the MBM (dashed line), Isomap (dashed-dotted line), and RME (solid line) methods.}
	\label{growth}
\end{figure}

\subsection{Gait cycle data}

For this application, we consider the data of angle measurements (in degrees) in the sagittal plane formed by the hip and knee of 39 children through a gait cycle, where time is measured in terms of the child's gait cycle such that every curve is given for values ranging between 0 and 1. The smoothed curves were obtained by fitting a Fourier basis system following the analysis of \citet{Ramsay-Silverman-05} for this data, where both sets of curves are periodic. Figure~\ref{gait} displays the curves of hip (on the left) and knee (on the right) angles observed during the gait. As we can see, a two-phase process can be identified for the knee motion, while for the hip motion there is a single-phase. Also, both sets of curves share a common pattern around which there are both phase and amplitude variability.

For this application, the template functions obtained by the Robust Manifold Estimator based on our algorithm seem to capture the salient features of the sample of hip and knee angle curves. Note also that the same template, located in the center of the samples, was chosen by all the estimators.

\begin{figure}[htb]
	\centering
	%height=0.68\textheight
	\includegraphics[height=\textwidth,width=0.6\textwidth,angle=270]{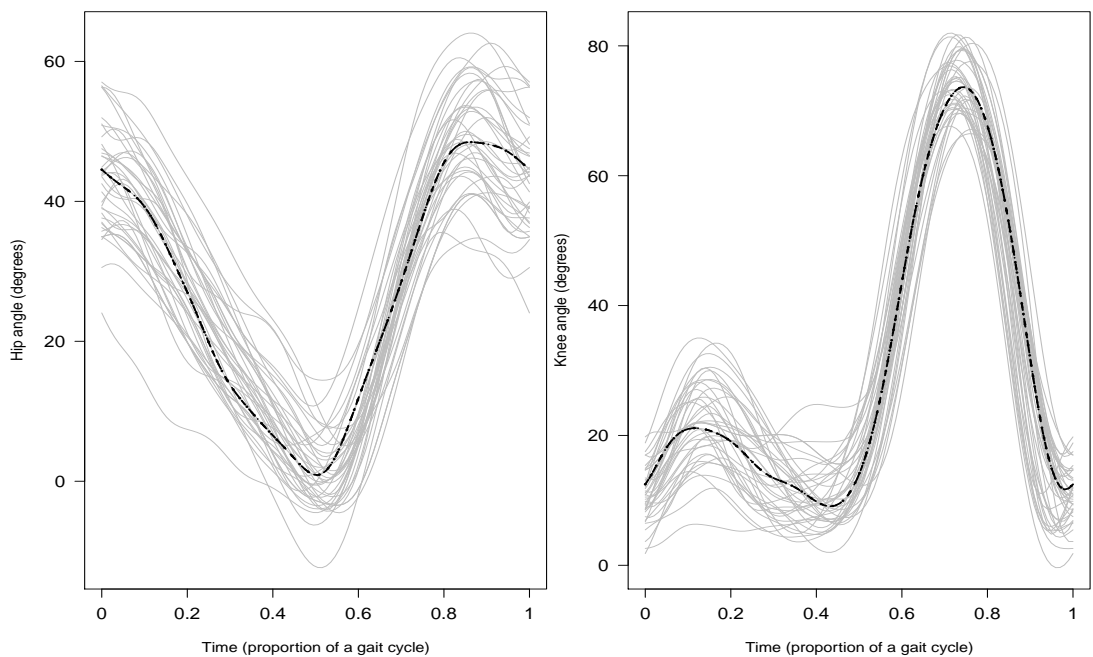}
	\caption[]{Angle curves formed by the hip (on the left) and knee (on the right) of 39 children through a gait cycle, and the MBM (dash-dotted line), Isomap (dotted line), and RME (dashed line) template estimators.}
	\label{gait}
\end{figure}

\subsection{Landscape reflectances data}

Finally, we consider two data sets where the corresponding observed curves represent the weekly reflectance profiles of two particular landscapes (corn and wheat). The reflectance is a measure of the incident electromagnetic radiation that is reflected by a given interface. For these data, there are 23 and 124 curves for corn and wheat landscapes respectively. The aim consists in extracting a representative curve of a type of landscape while observing the reflectance profiles of different landscapes of the same type. In Figure~\ref{reflectance}, the smoothed curves corresponding to reflectance patterns of two landscape types in the same region in the same period are showed. The smoothing was obtained from discrete data with B-spline basis system. The reflectance depends on the vegetation whose growth depends on the weather condition and the soil behavior. It is therefore relevant to consider that these profiles are deformations in translation, scale and amplitude of a single representative function of the reflectance behavior of each landscape type in this region at this time.

In Figure~\ref{reflectance}, we observe that for the corn landscape case, where there are relatively a few number of curves, the robust manifold estimator chooses a meaningful template curve which seems to appear at the center of curve sample, which coincides with the curve obtained by the modified band median estimator. The Isomap estimator chooses a different curve as representative function which is slightly away from the center of the sample. For the wheat landscape, all of three estimators choose a different template curve. Although all the estimated template curves follow the structural features of the sample, the RME estimator select a curve that is located more in the middle. In this application domain, extracting a curve by RME is best able to report data as reflecting their structure and thus to obtain a better representative and improve further future functional analysis.

%In Figure~\ref{reflectance}, we observe that all of three estimators choose a different curve as representative function for both landscapes. In the corn landscape case, where there are relatively a few number of curves, the robust manifold estimator chooses a meaningful template curve which seems to appear at the center of curve sample. The same conclusions can be drawn for the wheat landscape, where the local extrema are well represented. In this application domain, extracting a curve by RME is best able to report data as reflecting their structure and thus to obtain a better representative and improve further future functional analysis.

\begin{figure}[!htbp]
	\centering
	%height=0.68\textheight
	\includegraphics[height=\textwidth,width=0.6\textwidth,angle=270]{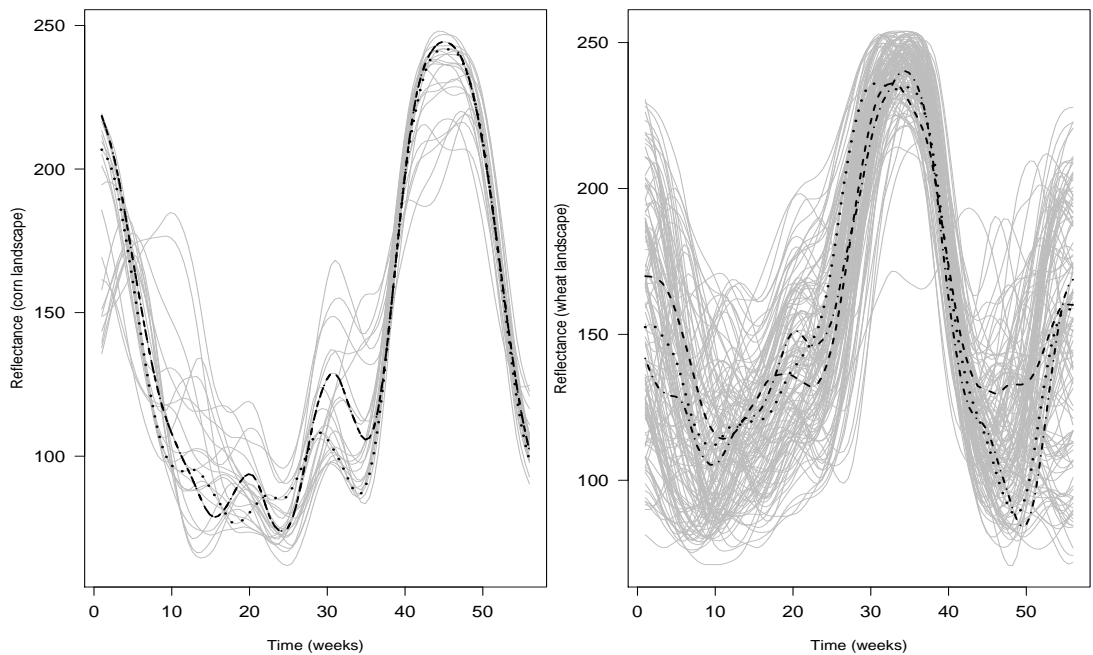}
	\caption[]{Reflectance curves of corn (left) and wheat (right) landscapes, and MBM (dash-dotted line), Isomap (dotted line), and RME (dashed line) template estimators.}
	\label{reflectance}
\end{figure}

\section{Concluding remarks}
\label{section6}

In this paper, we have proposed a robust algorithm to approximate the geodesic distance of the underlying
manifold. This approximated distance is used to build the corresponding empirical Fr\'{e}chet median of the functions. This function is a meaningful template curve for a sample of functions, which have both amplitude and time deformations.

Our approach relies on the fundamental paradigm of functional data analysis which involves treating the entire observed curve as a unit of observation rather than individual measurements from the curve. Indeed, we show that, when the structure of the deformations entails that the curve can be embedded into a manifold, finding a representative of a sample of curves corresponds to calculate an intrinsic statistic of observed curves on their unknown underlying manifold. Moreover in a translation model, i.e. where the curves are actually warped from an unknown pattern, both methodologies coincide since the structural median of a sample of curves corresponds to the intrinsic median on a one-dimensional manifold. Moreover, our methodology could be the first step for classification methods which require the choice of a good template, sharing the common properties of the dataset, see for instance in \citet{Shimizu-Mizuta-07}, \citet{Cuesta-Fraiman-07}, \citet{Sangalli.etal-10} and references therein.

From a computational point of view, our method is inspired by the ideas of the Isomap algorithm. We note that we have also used the Isomap algorithm in the simulation study and applications. Hence, our algorithm has the advantage of being parameter free and then it is of easiest use. One of the major drawbacks of these methodologies are that a relatively high number of data are required in order to guarantee a good approximation of the geodesic distance at the core of this work (see \citet{Tenenbaum-2000}). Nevertheless, our method is a robust version of the Isomap  and behaves better with respect to it. In addition it competes with a robust algorithm of the modified band median estimator. The \texttt{R} code is available at the webpage of one of the authors (\url{http://santiagogallongomez.wordpress.com/code/}) or upon request.

\section{Appendix}
\label{appendix}

\begin{proof} [Proof of Theorem~\ref{thmSM}]
Let $X$ be defined by
\[
\begin{split}
X\colon(b,c)&\to\mathbb{R}^m\\
	   a\;\;&\mapsto X(a)=\left(f\left(t_1-a\right),\ldots,f\left(t_m-a\right)\right)
\end{split}
\]
and set $\mathcal{C}=\left\{X(a)\in\mathbb{R}^m,\:a\in(b,c)\right\}$.

By assumption~\eqref{cond}, $\mathcal{C}$ can be seen as a submanifold of $\mathbb{R}^m$ of dimension 1 with corresponding geodesic distance defined by~\eqref{geodC}.

\noindent Take $\mu=X(\alpha)$ with $\alpha\in (b,c)$, thus we can write
\[
\begin{split}
\widehat{\mu}_I^1&=\argmin_{X(\alpha)\in\mathcal{C}}\sum_{i=1}^n\delta\left(X\left(A_i\right),X(\alpha)\right)\\
&=\argmin_{\mu\in\mathcal{C}}\sum_{i=1}^nD\left(A_i,\alpha\right)= \argmin_{\mu\in\mathcal{C}}C(\alpha)
\end{split}
\]
where $D$ is distance on $(b,c)$ given by
$$D\left(A_i,\alpha\right)=\left|\int_{A_i}^{\alpha}\left\|X^\prime(a)\right\|\mathrm{d}a\right|.$$
In the following, let $\left(A_{(i)}\right)_i$ be the ordered parameters. That is $A_{(1)}<\dots<A_{(n)}.$
Then, for a given $\alpha\in (b,c)$ such that $A_{(j)}<\alpha<A_{(j+1)}$, we get that
\[
\begin{split}
C(\alpha)&=jD\left(\alpha,A_{(j)}\right)+\sum_{i=1}^{j-1}iD\left(A_{(i)},A_{(i+1)}\right)\\
&\quad+(n-j)D\left(\alpha,A_{(j+1)}\right)+\sum_{i=j+1}^{n-1}(n-i)D\left(A_{(i)},A_{(i+1)}\right).
\end{split}
\]
For the sake of simplicity, let $n=2q+1$. It follows that $\widehat{\mathrm{med}}(A)=A_{(q+1)}$. Moreover, let $\alpha=A_{(j)}$ with $j<q+1$. By symmetry, the case $j>q+1$ holds. Then, we rewrite $C\left(\alpha\right)$ as
\begin{equation*}
C\left(\alpha\right)=\sum_{i=1}^{j-1}iD\left(A_{(i)},A_{(i+1)}\right)+\sum_{i=j}^{n-1}(n-i)D\left(A_{(i)},A_{(i+1)}\right)
\end{equation*}
and, by introducing $A_{(q+1)}$, we get that
\[
\begin{split}
C(\alpha)&=\sum_{i=1}^{j-1}iD\left(A_{(i)},A_{(i+1)}\right)+\sum_{i=j}^qiD\left(A_{(i)},A_{(i+1)}\right)\\
&\quad+\sum_{i=j}^q(n-2i)D\left(A_{(i)},A_{(i+1)}\right)+\sum_{i=q+1}^{n-1}(n-i)D\left(A_{(i)},A_{(i+1)}\right).
\end{split}
\]
Finally, we notice that
\begin{equation*}
C(\alpha)=C\left(A_{(q+1)}\right)+\sum_{i=j}^q(n-2i)D\left(A_{(i)},A_{(i+1)}\right)>C\left(A_{(q+1)}\right).
\end{equation*}
And the result follows since
$$\widehat{\mu}_I^1=\argmin_{\mu\in\mathcal{C}}C(\alpha)=X\left(A_{(q+1)}\right)=X\left(\widehat{\mathrm{med}}(A)\right)=\widehat{f}_\mathrm{SM}.$$
\end{proof}

\bibliography{newpaper.bib}
\end{document}